\documentclass[a4paper,reqno]{amsart}
\usepackage{amsmath,amssymb,amsthm,graphicx,mathrsfs,url}
\usepackage[usenames,dvipsnames]{color}
\definecolor{darkred}{rgb}{0.4,0.1,0.1}
\definecolor{darkblue}{rgb}{0.1,0.1,0.4}

\usepackage{tikz}
\usepackage[normalem]{ulem}

\numberwithin{equation}{section}
\theoremstyle{plain}% default
\newtheorem{thm}{Theorem}[section]

\newtheorem{lem}[thm]{Lemma}

\newtheorem{prop}[thm]{Proposition}
\newtheorem{cor}[thm]{Corollary}

\theoremstyle{remark}
\newtheorem{remark}[thm]{Remark}

\theoremstyle{definition}
\newtheorem{example}[thm]{Example}
\newtheorem{dfn}[thm]{Definition}
\newtheorem{hypothesis}[thm]{Hypothesis}

\DeclareMathOperator{\spann}{span}
\DeclareMathOperator{\diver}{div}
\DeclareMathOperator{\curl}{curl}
\DeclareMathOperator{\Real}{Re}
\DeclareMathOperator{\Imag}{Im}

\definecolor{darkgreen}{rgb}{0.1,0.45,0.1}
\definecolor{darkblue}{rgb}{0.1,0.1,0.4}
\definecolor{darkgrey}{rgb}{0.5,0.5,0.5}

\definecolor{darkred}{rgb}{0.6,0.0,0.0}

\newcommand\void[1]{}

      \def\dC{{\mathbb C}}

   \def\dN{{\mathbb N}}   
      \def\dR{{\mathbb R}}

\def\cS{{\mathcal S}}

\def\R{\mathbb{R}}
\def\C{\mathbb{C}}

\def\N{\mathbb{N}}

\newcommand{\dom}{\mathrm{dom}\,}

\allowdisplaybreaks

\def\dd{{\,\mathrm d}}
\def\ee{{\mathrm e}}

\newcounter{counter_a}

\title[Eigenvalue inequalities for mixed boundary conditions]{Eigenvalue inequalities for the Laplacian with mixed boundary conditions}

\author[V.~Lotoreichik]{Vladimir Lotoreichik}
\address{Department of Theoretical Physics,
Nuclear Physics Institute, Czech Academy of Sciences, 250 68, 
\v{R}e\v{z} near Prague, Czechia}
\email{lotoreichik@ujf.cas.cz}

\author[J.~Rohleder]{Jonathan Rohleder}
\address{TU Hamburg \\ Institut f\"ur Mathematik \\
Am Schwarzenberg-Campus~3 \\
Geb\"aude E \\
21073 Hamburg \\
Germany}
\email{jonathan.rohleder@tuhh.de}


\begin{document}

\begin{abstract}
Inequalities for the eigenvalues of the (negative) Laplacian subject to mixed boundary conditions on polyhedral and more general bounded domains are established. The eigenvalues subject to a Dirichlet boundary condition on a part of the boundary and a Neumann boundary condition on the remainder of the boundary are estimated in terms of either Dirichlet or Neumann eigenvalues. The results complement several classical inequalities between Dirichlet and Neumann eigenvalues due to P\'{o}lya, Payne, Levine and Weinberger, Friedlander, and others.
\end{abstract}

\maketitle

\section{Introduction}

Properties of Laplacian eigenvalues on bounded domains subject to various boundary conditions are a classical topic in spectral theory due to the fact that these eigenvalues can be interpreted, e.g., as frequencies of vibrating membranes; see Lord Rayleigh's famous book {\em The Theory of Sound}~\cite{Rayleigh}. A prominent line of research in this context is related to inequalities between Dirichlet and Neumann eigenvalues; its history dates back at least to the 1950s. On a bounded, sufficiently regular, connected domain $\Omega \subset \R^d$ denote by 
\begin{align*}
 0 < \lambda_1 < \lambda_2 \leq \lambda_3 \leq \cdots
\end{align*}
the eigenvalues of the (negative) Laplacian subject to a Dirichlet boundary condition on the boundary $\partial \Omega$ and by
\begin{align*}
 0 = \mu_1 < \mu_2 \leq \mu_3 \leq \cdots
\end{align*}
the eigenvalues corresponding to a Neumann condition. A variational argument easily implies $\mu_k \leq \lambda_k$ for all $k \in \N$, but in fact several non-trivial improvements of this inequality were found in the course of time. In 1952 P\'olya~\cite{P52} proved $\mu_2 <\lambda_1$ in the two-dimensional case, see also Szeg\H{o}'s contribution~\cite{S54}. Shortly after, in 1955 Payne~\cite{P55} showed $\mu_{k + 2} < \lambda_k$ for all $k \in \N$ for convex, two-dimensional $\Omega$ with $C^2$-boundary. This result was extended and generalized three decades later by Levine and Weinberger~\cite{LW86}, who (amongst other estimates) obtained the inequality $\mu_{k + d} \leq \lambda_k$ for all $k \in \N$, for arbitrary convex domains. For not necessarily convex bounded $C^1$-domains the inequality $\mu_{k + 1} \leq \lambda_k$ for all $k \in \N$ was established by Friedlander~\cite{F91} in 1991, which is valid for all $d$. In 2004 Filonov~\cite{F05} showed that even $\mu_{k + 1} < \lambda_k$ for all $k \in \N$ holds in every space dimension $d \geq 2$ and for every bounded Lipschitz (and even more general) domain.

The present paper focuses on Laplacian eigenvalues for the mixed case of a Dirichlet boundary condition on a nonempty part $\Gamma = \Gamma_{\rm D}$ of $\partial \Omega$ and a Neumann condition on the complement $\Gamma_{\rm N}$ of $\Gamma_{\rm D}$ in $\partial \Omega$. These boundary conditions are ``in between'' the Neumann and Dirichlet problems in the sense that the corresponding eigenvalues
\begin{align*}
 0 < \lambda_1^\Gamma < \lambda_2^\Gamma \leq \lambda_3^\Gamma \leq \cdots
\end{align*}
satisfy 
\begin{align}\label{eq:trivial}
 \mu_k \leq \lambda_k^\Gamma \leq \lambda_k \quad \text{for~all}~k \in \N;
\end{align}
this is a trivial consequence of variational principles. Our aim here is to investigate the position of the eigenvalues of the mixed problem in comparison with the Neumann and Dirichlet eigenvalues in more detail. In general this position will depend on the size of the Dirichlet and Neumann parts $\Gamma_{\rm D}$ and $\Gamma_{\rm N}$, respectively, as well as on their geometries. In  this paper we study the case of Lipschitz domains which are polygonal or polyhedral or have some polyhedra-like properties.

In our first main result, Theorem~\ref{thm:NeumannMixed}, we provide an improvement of the first inequality in~\eqref{eq:trivial} comparing Neumann and mixed Laplacian eigenvalues. Here we assume that $\Omega$ is a bounded Lipschitz domain and that the ``Neumann part'' $\Gamma_{\rm N}$ of the boundary is small enough in the sense that there exists a nontrivial vector being tangential to almost all points in $\Gamma_{\rm N}$. Under these conditions we obtain the inequality
\begin{align}\label{eq:INeumann}
 \mu_{k + 1} \leq \lambda_k^\Gamma \quad \text{for~all}~k \in \N.
\end{align}
This result applies to several configurations. For instance, the assumptions of the theorem are satisfied if $\Gamma_{\rm N}$ is a part of $\partial \Omega$ having zero curvature into at least one direction or if, in three or more space dimensions, $\Gamma_{\rm N}$ consists of two flat parts of the boundary, e.g.\ two faces of a polyhedron; cf.\ the corollaries and examples in Section~\ref{sec:Neumann}.

Our second main result deals with the comparison of mixed and Dirichlet eigenvalues, aiming at an improvement of the second inequality in~\eqref{eq:trivial}. Due to the methods of proof used in this part of the paper we restrict ourselves to the case that $\Omega$ is a polygonal (for $d = 2$) or polyhedral (for $d \geq 3$), convex domain. Letting $l$ be the number of linearly independent vectors which are tangential to almost all points of $\Gamma_{\rm D}$, in Theorem~\ref{thm:DirichletMixed} we show the inequality
\begin{align}\label{eq:DIdentity}
 \lambda_{k + l}^\Gamma \leq \lambda_k \quad \text{for~all}~k \in \N.
\end{align}
If, for instance, $\Gamma_{\rm D}$ is one face of the polyhedral domain $\Omega \subset \R^d$ then~\eqref{eq:DIdentity} implies
\begin{align*}
 \lambda_{k + d - 1}^\Gamma \leq \lambda_k \quad \text{for~all}~k \in \N.
\end{align*}
Thus the comparison of mixed and Dirichlet eigenvalues exhibits some dimension dependence similar to the comparison of Neumann and Dirichlet eigenvalues found in~\cite{LW86}. On the other hand, if $\Gamma_{\rm D}$ consists of at most $d - 1$ pairwise non-parallel faces then~\eqref{eq:DIdentity} yields
\begin{align*}
 \lambda_{k + 1}^\Gamma \leq \lambda_k \quad \text{for~all}~k \in \N.
\end{align*}
For further consequences of Theorem~\ref{thm:DirichletMixed} we refer the reader to the corollaries in Section~\ref{sec:Dirichlet}.

We point out that in general none of the inequalities~\eqref{eq:INeumann} and~\eqref{eq:DIdentity} is strict. This can be seen from simple examples of a square or a cube and proper choices of the Dirichlet and Neumann parts of the boundary, see Example~\ref{ex:square} and Example~\ref{ex:cube} below. However, under additional assumptions on the choice of $\Gamma_{\rm D}$ and $\Gamma_{\rm N}$ strict inequality can be obtained, see Corollary~\ref{cor:strict} and Corollary~\ref{cor:DirichletMixedOneFace}.

The proofs of our main results are based on variational principles and proper choices of test functions. For the proof of~\eqref{eq:INeumann} we choose an exponential function suitable to the joint tangent vector of $\Gamma_{\rm N}$; cf.~\cite{F05} for the use of an exponential test function in the comparison of Neumann and Dirichlet eigenvalues. For the proof of~\eqref{eq:DIdentity} we employ appropriate linear combinations of derivatives of Dirichlet eigenfunctions as test functions; this is motivated by~\cite{LW86}. However, our calculations differ essentially from those made in~\cite{LW86} as the mentioned work makes use of differential geometric tools and curvature properties of the boundary while the proof of our Theorem~\ref{thm:DirichletMixed} relies on an integral identity for polyhedral domains (which fails for general, curved domains). For $d = 2$ this identity is contained in Grisvard's classical book~\cite{G85}; in the appendix of the present work we provide a proof of it for arbitrary dimensions.

Finally, let us mention that eigenvalue inequalities for Laplacians and more general elliptic operators where studied recently in~\cite{AM12, BRS16,FL10,GM09,K10, LR15, R14}. Especially inequalities for Laplacian eigenvalues of particular polygonal domains like triangles and rhombi have attracted interest recently due to applications to the hot spots conjecture and other problems, see, e.g.,~\cite{S15,S16}. For further literature on mixed elliptic boundary value problems (sometimes also called Zaremba problems) we refer the reader to~\cite{A11,B94,G11,P16,S02,S68}. For elliptic boundary value problems on polygonal and polyhedral domains see the monographs~\cite{D88, G85,MR10}.

\section{Preliminaries: Laplacian eigenvalue problems with mixed boundary conditions}
\label{sec:prelim}

Let us first fix some notation and recall some basic facts. Throughout the whole paper $\Omega \subset \R^d$, $d \geq 2$, is a bounded, connected Lipschitz domain. Recall that by Rademacher's theorem for almost all $x' \in \partial \Omega$ there exists a well-defined outer unit normal vector $\nu (x')$. Consequently, the $(d - 1)$-dimensional tangential hyperplane
\begin{align}\label{eq:tangential}
 T_{x'} = \bigg\{\tau = (\tau_1, \dots, \tau_d)^\top \in \R^d : \sum_{j = 1}^d \tau_j \nu_j (x') = 0 \bigg\}
\end{align}
can be defined for almost all $x' \in \partial \Omega$. We denote by $H^k (\Omega)$ the Sobolev spaces of orders $k \geq 1$ on $\Omega$ and by $H^s (\partial \Omega)$ the Sobolev spaces of orders $s \in [- 1/2, 1/2]$ on $\partial \Omega$; in particular, for $s \in [0, 1/2]$ the space $H^{- s} (\partial \Omega)$ is the dual of $H^{s} (\partial \Omega)$. For $u \in H^1 (\Omega)$ we denote by $u |_{\partial \Omega} \in H^{1/2} (\partial \Omega)$ the trace of $u$. Moreover, if $u \in H^1 (\Omega)$ with $\Delta u \in L^2 (\Omega)$ distributionally then the normal derivative $\partial_\nu u |_{\partial \Omega} \in H^{- 1/2} (\partial \Omega)$ of $u$ on $\partial \Omega$ can be defined via Green's identity
\begin{align}\label{eq:Green}
 \int_\Omega \nabla u \cdot \overline{\nabla v} \dd x = - \int_\Omega \Delta u \overline{v} \dd x + (\partial_\nu u |_{\partial \Omega}, v |_{\partial \Omega})_{\partial \Omega}, \quad v \in H^1 (\Omega),
\end{align}
where $(\cdot, \cdot)_{\partial \Omega}$ is the (sesquilinear) duality between $H^{- 1/2} (\partial \Omega)$ and $H^{1/2} (\partial \Omega)$; cf., e.g.,~\cite[Chapter~4]{McL}. If $u$ is sufficiently regular up to the boundary, for instance $u \in H^2 (\Omega)$, then $\partial_\nu u |_{\partial \Omega} = \nu \cdot \nabla u |_{\partial \Omega}$ almost everywhere on $\partial \Omega$; in this case the duality in~\eqref{eq:Green} turns into the boundary integral of $\partial_\nu u |_{\partial \Omega} \overline{v |_{\partial \Omega}}$ with respect to the standard surface measure on $\partial \Omega$. In the following, for a relatively open subset $\omega$ of $\partial \Omega$ we write $\partial_\nu u |_{\omega} = 0$ if
\begin{align}\label{eq:partNeumann}
 (\partial_\nu u |_{\partial \Omega}, v |_{\partial \Omega} )_{\partial \Omega} = 0 \quad \text{for~all}~v \in 
 H^1 (\Omega)~\text{such~that}~v |_{\partial \Omega \setminus \overline \omega} = 0.
\end{align}
Note that for $u$ being sufficiently regular in a neighborhood of $\omega$ the condition~\eqref{eq:partNeumann} simply means $(\nu \cdot \nabla u) |_\omega = 0$.

In order to write down the mixed Dirichlet--Neumann eigenvalue problem, we make the following assumptions.

\begin{hypothesis}\label{hyp}
We assume that $\Gamma = \Gamma_{\rm D}$ and $\Gamma_{\rm N}$ are two relatively open, non\-empty subsets of $\partial \Omega$ such that $\Gamma_{\rm D} \cap \Gamma_{\rm N} = \emptyset$ and $\partial \Omega \setminus (\Gamma_{\rm D} \cup \Gamma_{\rm N})$ has measure zero.
\end{hypothesis}

Under the assumption of Hypothesis~\ref{hyp} we define
\begin{align*}
 H_{0, \Gamma}^1 (\Omega) = \left\{u \in H^1 (\Omega) : u |_{\Gamma} = 0 \right\},
\end{align*}
the space of functions in $H^1 (\Omega)$ whose trace vanishes on $\Gamma$. The negative Laplacian subject to a Dirichlet boundary condition on $\Gamma = \Gamma_{\rm D}$ and a Neumann boundary condition on $\Gamma_{\rm N}$ is given by
\begin{align*}
 - \Delta_\Gamma u = - \Delta u, \quad \dom (- \Delta_\Gamma) = \left\{ u \in H_{0, \Gamma}^1 (\Omega) : \Delta u \in L^2 (\Omega), \partial_\nu u |_{\Gamma_{\rm N}} = 0 \right\}.
\end{align*}
The operator $- \Delta_\Gamma$ is selfadjoint in $L^2 (\Omega)$ and has a purely discrete spectrum. In fact, $- \Delta_\Gamma$ corresponds to the closed, nonnegative, symmetric sesquilinear form $\{u, v\} \mapsto \int_\Omega \nabla u \cdot \overline{\nabla v} \dd x$ with domain $H_{0, \Gamma}^1 (\Omega)$; cf.~\cite{BS87,Kato,S12} for more details on semi-bounded selfadjoint operators and corresponding quadratic forms. Therefore the eigenvalues of the mixed Laplacian $- \Delta_\Gamma$, ordered nondecreasingly and counted with multiplicities, are given by the min-max principle
\begin{align}\label{eq:minmax}
 \lambda_k^\Gamma = \min_{\substack{L \subset H_{0, \Gamma}^1 (\Omega) \\ \dim L = k}} \,\, \max_{u \in L \setminus \{0\}} \frac{\int_\Omega |\nabla u|^2 \dd x}{\int_\Omega |u|^2 \dd x}, \quad k \in \N.
\end{align}
As is well-known, the eigenvalues of the selfadjoint Laplacian with a Neumann boundary condition on the whole boundary $\partial \Omega$ are given by
\begin{align}\label{eq:minMaxNeumann}
 \mu_k = \min_{\substack{L \subset H^1 (\Omega) \\ \dim L = k}} \,\, \max_{u \in L \setminus \{0\}} \frac{\int_\Omega |\nabla u|^2 \dd x}{\int_\Omega |u|^2 \dd x}, \quad k \in \N.
\end{align}

In the following we provide a first, simple observation on the behavior of the eigenvalues of the mixed problem when the Dirichlet part of the boundary is increased. The next, preparatory lemma is a simple consequence of a unique continuation principle; it can be proven similar to~\cite[Lemma~3.1]{R14}.

\begin{lem}\label{lem:UCP}
Let $\Omega$ be a bounded, connected Lipschitz domain, let $\lambda \in \R$ and let $u \in H^1 (\Omega)$ be such that $- \Delta u = \lambda u$. If $\omega \subset \partial \Omega$ is a relatively open, nonempty set such that $u |_\omega = 0$ and $\partial_\nu u |_\omega = 0$ then $u = 0$ identically on $\Omega$.
\end{lem}

The previous lemma can be used to derive the following strict monotonicity principle, which will be used in the following sections.

\begin{prop}\label{prop:monotonicity}
Assume that $\Gamma \subset \Gamma' \subset \partial \Omega$ are nonempty, relatively open sets such that $\Gamma' \setminus \Gamma$ has a nontrivial interior. Then
\begin{align*}
 \lambda_k^\Gamma < \lambda_k^{\Gamma'}
\end{align*}
holds for all $k \in \N$.
\end{prop}

\begin{proof}
Let $k \in \N$ and $\lambda = \lambda_k^{\Gamma'}$. By the min-max principle~\eqref{eq:minmax} there exists a subspace $L \subset H_{0, \Gamma'}^1 (\Omega)$ with $\dim L = k$ such that
\begin{align*}
 \int_\Omega |\nabla u|^2 \dd x \leq \lambda \int_\Omega |u|^2 \dd x, \quad u \in L.
\end{align*}
Hence for all $u \in L$ and all $v \in \ker (- \Delta_\Gamma - \lambda)$ we have $u + v \in H_{0, \Gamma}^1 (\Omega)$ and
\begin{align}\label{eq:usualCalc}
\begin{split}
 \int_\Omega |\nabla (u + v)|^2 \dd x & = \int_\Omega |\nabla u|^2 \dd x + 2 \Real \int_\Omega \nabla v \cdot \overline{\nabla u} \dd x + \int_\Omega |\nabla v|^2 \dd x \\
 & \leq \lambda \int_\Omega |u|^2 \dd x + 2 \lambda \Real \int_\Omega v \overline{u} \dd x + \lambda \int_\Omega |v|^2 \dd x \\
 & = \lambda \int_\Omega |u + v|^2 \dd x,
\end{split}
\end{align}
where we have used Green's identity~\eqref{eq:Green} as well as $u |_\Gamma = 0$ and $\partial_\nu v |_{\partial \Omega \setminus \overline \Gamma} = 0$. Moreover, $L \cap \ker (- \Delta_\Gamma - \lambda) = \{0\}$, which follows from Lemma~\ref{lem:UCP} when choosing $\omega$ to be the interior of $\Gamma' \setminus \Gamma$. Thus
\begin{align*}
 \dim \big( L + \ker (- \Delta_\Gamma - \lambda) \big) = k + \dim \ker (- \Delta_\Gamma - \lambda)
\end{align*}
and with~\eqref{eq:usualCalc} and the min-max principle it follows
\begin{align}\label{eq:dasIsses}
 \lambda_k^\Gamma \leq \lambda_{k + \dim \ker (- \Delta_\Gamma - \lambda)}^\Gamma \leq \lambda.
\end{align}
Since $\lambda_k^\Gamma = \lambda$ together with~\eqref{eq:dasIsses} would imply $\lambda_k^\Gamma = \lambda_{k + \dim \ker (- \Delta_\Gamma - \lambda)}^\Gamma = \lambda$, i.e., $\lambda$ is an eigenvalue of $-\Delta_\Gamma$ of multiplicity $\dim \ker (- \Delta_\Gamma - \lambda) + 1$ or larger, a contradiction, it follows $\lambda_k^\Gamma < \lambda = \lambda_k^{\Gamma'}$.
\end{proof}

Polygonal and (multidimensional) polyhedral domains play an important role in the following sections. In order to avoid ambiguities we give the following definition.

\begin{dfn}\label{def:poly}
Let $\Omega \subset \R^d$, $d \geq 2$, be a bounded, connected Lipschitz domain. 
\begin{enumerate}
 \item If $d = 2$ we say that $\Omega$ is a polyhedral (or polygonal) domain if $\partial \Omega$ is the union of finitely many line segments.
 \item Recursively, if $d \geq 3$ we say that $\Omega$ is a polyhedral domain if for each ($d - 1$)-dimensional affine hyperplane $H \subset \R^d$ the intersection $H \cap \Omega$ is either a polyhedral domain in $\R^{d - 1}$ (where we identify $H$ with $\R^{d - 1}$) or empty.
\end{enumerate}
\end{dfn}

Note that in the case $d = 3$ a bounded Lipschitz domain is polyhedral if and only if its boundary is the union of finitely many polygonal faces.

\section{Inequalities for Neumann and mixed eigenvalues}\label{sec:Neumann}

In this section we compare Neumann and mixed Laplacian eigenvalues for polyhedral and more general domains in any space dimension $d \geq 2$. We assume that $\Omega \subset \R^d$, $d \geq 2$, is a bounded, connected Lipschitz domain and that Hypothesis~\ref{hyp} is satisfied. As before we denote by $0 = \mu_1 < \mu_2 \leq \mu_2 \leq \dots$ the Neumann Laplacian eigenvalues and by $\lambda_1^\Gamma < \lambda_2^\Gamma \leq \lambda_3^\Gamma \leq \dots$ the eigenvalues of $- \Delta_\Gamma$. 

For the following theorem recall that the tangential hyperplane $T_{x'}$ exists for almost all $x' \in \partial \Omega$; cf.~\eqref{eq:tangential}. We define $\hat \Gamma_{\rm N}$ to be the set of all $x' \in \Gamma_{\rm N}$ such that $T_{x'}$ exists. We define the linear subspace
\begin{align*}
 \cS(\Gamma_{\rm N}) := \bigcap_{x' \in \hat \Gamma_{\rm N}} T_{x'}
\end{align*}
of $\R^d$ consisting of all vectors being tangential to all $x' \in \Gamma_{\rm N}$ apart from a set of measure zero. With this notation the main result of this section looks as follows.

\begin{thm}\label{thm:NeumannMixed}
Let $\Omega \subset \R^d$, $d \geq 2$, be a bounded, connected Lipschitz domain and let Hypothesis~\ref{hyp} be satisfied. If $\dim \cS(\Gamma_{\rm N}) \geq 1$ then 
\begin{align}\label{eq:Jawoll}
 \mu_{k+1} \leq \lambda_k^\Gamma
\end{align}
holds for all $k\in\dN$.
\end{thm}

\begin{proof}
Let $k \in \N$ and define $\lambda = \lambda_k^\Gamma > 0$. By the min-max principle~\eqref{eq:minmax} there exists a subspace $L$ of $H_{0, \Gamma}^1 (\Omega)$ such that $\dim L = k$ and
\begin{align}\label{eq:EigenspaceMinMax}
 \int_\Omega |\nabla u|^2 \dd x \leq \lambda \int_\Omega |u|^2 \dd x
\end{align}
holds for all $u \in L$. Due to the assumption $\dim \cS (\Gamma_{\rm N}) \geq 1$ there exists a vector $\omega_0 \in \cS (\Gamma_{\rm N})$ such that $|\omega_0| = \sqrt{\lambda}$ holds. Letting $v (x) = e^{i \omega_0 \cdot x}$, $x \in \Omega$, we have $v \in H^2 (\Omega)$, $\nabla v = i \omega_0 v$, and $- \Delta v = \lambda v$. With the help of~\eqref{eq:EigenspaceMinMax}, for each $u \in L$ and each $c \in \dC$ we obtain
\begin{align}\label{eq:anfang}
\begin{split}
 \int_\Omega |\nabla (u + c v)|^2 \dd x & = \int_\Omega |\nabla u|^2 \dd x + 2 \Real \int_\Omega c \nabla v \cdot \overline{\nabla u} \dd x + \int_\Omega |c \nabla v|^2 \dd x \\
 & \leq \lambda \int_\Omega |u|^2 \dd x + 2 \Real \int_\Omega c \nabla v \cdot \overline{\nabla u} \dd x + \lambda \int_\Omega |c v|^2 \dd x.
\end{split}
\end{align}
Moreover, Green's identity~\eqref{eq:Green} together with $u |_{\Gamma_{\rm D}} = 0$ and $\omega_0 \cdot \nu |_{\Gamma_{\rm N}} = 0$ yields
\begin{align}\label{eq:ende}
\begin{split}
 \int_\Omega \nabla v \cdot \overline{\nabla u} \dd x & = - \int_\Omega \Delta v \overline{u} \dd x + (\partial_\nu v |_{\partial \Omega}, u |_{\partial \Omega})_{\partial \Omega} \\
 & = \lambda \int_\Omega v \overline{u} \dd x + \int_{\partial \Omega} i v \overline{u} \omega_0 \cdot \nu \dd \sigma \\
 & = \lambda \int_\Omega v \overline{u} \dd x,
\end{split}
\end{align}
where $\sigma$ is the standard surface measure on $\partial \Omega$. Combining~\eqref{eq:anfang} and~\eqref{eq:ende} we arrive at
\begin{align}\label{eq:letzte}
 \int_\Omega |\nabla (u + c v)|^2 \dd x & \leq \lambda \int_\Omega |u + c v|^2 \dd x 
\end{align}
for all $u \in L$ and all $c \in \C$. Moreover, the function $v$ does not belong to $L$ as all functions in $L$ vanish on $\Gamma$. Hence $\dim (L + \spann \{v\}) = k + 1$ and~\eqref{eq:letzte} together with the min-max principle~\eqref{eq:minMaxNeumann} implies the assertion of the theorem.
\end{proof}

The following corollaries are direct consequences of Theorem~\ref{thm:NeumannMixed}. They illustrate the application of Theorem~\ref{thm:NeumannMixed} to domains with partially flat boundaries. 

\begin{cor}\label{cor:eins}
Let $\Omega$ be a bounded, connected Lipschitz domain in $\R^2$ and assume that $\Gamma_{\rm N}$ is contained in the union of parallel line segments. Then~\eqref{eq:Jawoll} holds for all $k\in\dN$.
\end{cor}

\begin{cor}\label{cor:zwei}
Let $\Omega$ be a bounded, connected Lipschitz domain in $\R^3$ and assume that $\Sigma \subset \partial \Omega$ is the union of two plane parts and all plane parts of the boundary which are parallel to one of these two. If $\Gamma_{\rm N} \subset \Sigma$ then~\eqref{eq:Jawoll} holds for all $k\in\dN$.
\end{cor}

The domains $\Omega_1$ and $\Omega_2$ in Figure~\ref{fig:Neumann} are examples to which the previous corollaries apply.

\begin{figure}[h]
\begin{tikzpicture}
\pgfsetlinewidth{0.8pt}
\color{gray}
\pgfputat{\pgfxy(2.3,-0.4)}{\pgfbox[center,base]{$\Gamma_{\rm N}$}}
\color{gray}
\pgfxyline(-0.7,-0.8)(-1.5,0)
\pgfxyline(1.6,0)(2.6,-1)
\color{black}
\pgfxyline(-1.5,0)(1.6,0)
\pgfxyline(2.6,-1)(1.3,-1.3)
\draw (1.3,-1.3) arc [radius=2, start angle=45, end angle= 107];
\pgfputat{\pgfxy(1.2,-0.6)}{\pgfbox[center,base]{$\Omega_1$}}
\end{tikzpicture} 
\qquad \qquad
\begin{tikzpicture}
\pgfsetlinewidth{0.8pt}
\fill[color=lightgray] (-0.5,-1) -- (0.5,-1) -- (0.5, -0.5) -- (-0.5,-0.5) -- (-0.5,-1);
\fill[color=lightgray] (-1,0) -- (1,0) -- (1, 0.5) -- (-1,0.5) -- (-1,0);
\fill[color=lightgray] (-0.5,-1) -- (-1,0) -- (-1,0.5) -- (-0.5,-0.5) -- (-0.5,-1);
\draw[densely dashed] (-1,0) -- (1,0);
\draw (1,0) -- (0.5,-1) -- (-0.5, -1) -- (-1,0);
\draw (-1,0.5) -- (1,0.5) -- (0.5,-0.5) -- (-0.5, -0.5) -- (-1,0.5);
\draw (-1,0) -- (-1,0.5);
\draw (1,0) -- (1,0.5);
\draw (-0.5,-1) -- (-0.5,-0.5);
\draw (0.5,-1) -- (0.5,-0.5);
\pgfputat{\pgfxy(0.2,-0.35)}{\pgfbox[center,base]{$\Omega_2$}}
\color{gray}
\pgfputat{\pgfxy(-1.3,-0.4)}{\pgfbox[center,base]{$\Gamma_{\rm N}$}}
\end{tikzpicture}
\qquad \qquad
\begin{tikzpicture}
\pgfsetlinewidth{0.8pt}
\fill[color=black] (-0.6,0) circle (0.6cm and 0.15cm);
\fill[color=black] (-0.6,1.2) circle (0.6cm and 0.15cm);
\fill[left color=lightgray!50!black,right color=lightgray!50!black,middle color=lightgray!50,shading=axis,opacity=0.25] (0,0) -- (0,1.2) arc (360:180:0.6cm and 0.15cm) -- (-1.2,0) arc (180:360:0.6cm and 0.15cm);
\draw (-1.2,1.2) -- (-1.2,0) arc (180:360:0.6cm and 0.15cm) -- (0,1.2) ++ (-0.6,0) circle (0.6cm and 0.15cm);
\draw[densely dashed] (-1.2,0) arc (180:0:0.6cm and 0.15cm);
\pgfputat{\pgfxy(-0.5,0.4)}{\pgfbox[center,base]{$\Omega_3$}}
\pgfputat{\pgfxy(0.7,0.5)}{\pgfbox[center,base]{$\Gamma_{\rm D}$}}
\pgfsetlinewidth{0.3pt}
\pgfxyline(0.4,0.5)(0.1,0.1)
\pgfxyline(0.4,0.8)(0.1,1.1)
\pgfsetlinewidth{0.8pt}
\color{gray}
\pgfputat{\pgfxy(-1.6,0.5)}{\pgfbox[center,base]{$\Gamma_{\rm N}$}}
\pgfsetlinewidth{0.8pt}
\color{gray}
\pgfxyline(-1.2,1.2)(-1.2,0)
\pgfxyline(0,1.2)(0,0)
\end{tikzpicture}
\caption{Three configurations for which the inequality~\eqref{eq:Jawoll} holds; cf.\ Corollary~\ref{cor:eins}--\ref{cor:zwei} and Example~\ref{ex:cylinder}. 
}
\label{fig:Neumann}
\end{figure}
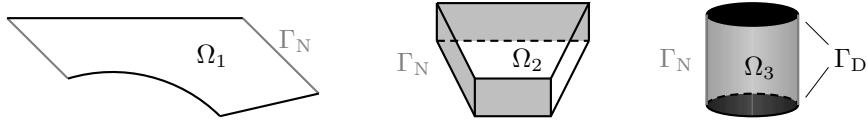

The next example shows that Theorem~\ref{thm:NeumannMixed} can also be applied to non-polyhedral three-dimensional domains.

\begin{example}\label{ex:cylinder}
Let $\Omega \subset \R^3$ be a cylinder with possibly deformed top and bottom faces. Moreover, assume that $\Gamma_{\rm N}$ is contained in the shell of $\Omega$. Then $\dim \cS (\Gamma_{\rm N}) = 1$ and Theorem~\ref{thm:NeumannMixed} implies~\eqref{eq:Jawoll} for all $k\in\dN$. For the simplest case of a non-deformed cylinder see the domain $\Omega_3$ in Figure~\ref{fig:Neumann}.
\end{example}

Theorem~\ref{thm:NeumannMixed} asserts that the inequality $\mu_{k + 1} \leq \lambda_k^\Gamma$ holds for all $k$ if $\Gamma_{\rm N}$ is not too large in a certain sense. The following example shows that the eigenvalue inequality is violated if $\Gamma_{\rm N}$ is too large.

\begin{example}
Consider the domain $\Omega := [0,\pi]^2 \subset \dR^2$ and set $\Gamma_{\rm D} := (0,\pi) \times \{0\}$, i.e., we impose a Dirichlet boundary condition on one side of the square $\Omega$ and Neumann boundary conditions on the rest of the boundary; in this case $\dim \cS (\Gamma_{\rm N}) = 0$. The Laplacian eigenfunctions and eigenvalues corresponding to the mixed and the pure Neumann problem on $\Omega$ can be calculated explicitly using separation of variables. For the mixed problem the eigenvalues are given by the numbers $(n - 1)^2 + (m - 1/2)^2$ with $n, m \in \N$, while the eigenvalues of the pure Neumann problem are $(n - 1)^2 + (m - 1)^2$ with $n, m \in \N$. In particular,
\begin{align*}
 \mu_2 = 1 > 1/4 = \lambda_1^\Gamma,
\end{align*}
so that the inequality~\eqref{eq:Jawoll} fails already for $k = 1$.
\end{example}

The following example shows that in general no strict inequality holds in the situation of Theorem~\ref{thm:NeumannMixed}.

\begin{example}\label{ex:square}
Let again $\Omega = [0, \pi]^2 \subset \R^2$ and let $\Gamma_{\rm D} = (0, \pi) \times \{0, \pi\}$ consist of two parallel faces. Then the Neumann eigenvalues are $(n - 1)^2 + (m - 1)^2$ and the mixed eigenvalues are $(n - 1)^2 + m^2$, $n, m \in \N$, yielding
\begin{align*}
 \mu_2 = 1 = \lambda_1^\Gamma.
\end{align*}
However, $\dim \cS (\Gamma_{\rm N}) = 1$, i.e., the assumptions of Theorem~\ref{thm:NeumannMixed} are satisfied.
\end{example}

Theorem~\ref{thm:NeumannMixed} can be combined with Proposition~\ref{prop:monotonicity} implying the following result. Roughly speaking, it states that the inequality~\eqref{eq:Jawoll} is strict if $\Gamma_{\rm N}$ can be enlarged nontrivially such that the condition on the dimension of the joint tangential space is not violated.

\begin{cor}\label{cor:strict}
Let $\Omega \subset \R^d$, $d \geq 2$, be a bounded, connected Lipschitz domain and let Hypothesis~\ref{hyp} be satisfied. Moreover, let $\Gamma' = \Gamma_{\rm D}'$ and $\Gamma_{\rm N}'$ be relatively open, nonempty subsets of $\partial \Omega$ such that $\Gamma_{\rm D}' \cap \Gamma_{\rm N}' = \emptyset$, $\partial \Omega \setminus (\Gamma_{\rm D}' \cup \Gamma_{\rm N}')$ has measure zero, and $\Gamma' \subset \Gamma$. If $\Gamma \setminus \Gamma'$ has a nonempty interior and $\dim \cS (\Gamma_{\rm N}') \geq 1$ then
\begin{align}\label{eq:nochmalStrict}
 \mu_{k+1} < \lambda_k^\Gamma
\end{align}
holds for all $k\in\dN$.
\end{cor}

We provide an exemplary application of Corollary~\ref{cor:strict} in the next example.

\begin{example}
Let $\Omega \subset \R^3$ be a polyhedral domain whose boundary contains two parallel faces $\Gamma_1, \Gamma_2$. If we choose $\Gamma_{\rm N} = \Gamma_1$ and $\Gamma = \Gamma_{\rm D}$ contains all faces of $\Omega$ except $\Gamma_1$ then the assumptions of Corollary~\ref{cor:strict} are satisfied with $\Gamma_{\rm N}' = \Gamma_1 \cup \Gamma_2$. Hence~\eqref{eq:nochmalStrict} is satisfied.
\end{example}

\section{Inequalities for Dirichlet and mixed eigenvalues on polygonal and polyhedral domains}\label{sec:Dirichlet}

In this section we provide inequalities which compare the eigenvalues $\lambda_1^\Gamma < \lambda_2^\Gamma \leq \lambda_3^\Gamma \leq \dots$ of the operator $- \Delta_\Gamma$ subject to mixed boundary conditions with the eigenvalues $\lambda_1 < \lambda_2 \leq \lambda_3 \leq\dots$ of the Dirichlet Laplacian. Throughout this section we make an additional restriction on the class of domains. We assume that $\Omega \subset \R^d$, $d \geq 2$, is a polyhedral, convex, bounded domain; cf.\ Definition~\ref{def:poly}. Moreover, we assume that $\Gamma_{\rm D}$ and $\Gamma_{\rm N}$ are chosen according to Hypothesis~\ref{hyp}. For the main result of this section let $\hat \Gamma_{\rm D}$ denote the set of points $x' \in \Gamma_{\rm D}$ such that the tangential hyperplane $T_{x'}$ exists, see~\eqref{eq:tangential}, and define the linear subspace
\begin{align*}
 \cS (\Gamma_{\rm D}) = \bigcap_{x' \in \hat \Gamma_{\rm D}} T_{x'}
\end{align*}
of $\R^d$ consisting of all vectors being tangential to almost all points of $\Gamma_{\rm D}$. Note that $\dim \cS (\Gamma_{\rm D}) \in \{0, \dots, d - 1\}$. The main result of this section reads as follows; its proof relies heavily on Lemma~\ref{lem:det} in the appendix.

\begin{thm}\label{thm:DirichletMixed}
Let Hypothesis~\ref{hyp} be satisfied and assume, in addition, that $\Omega$ is polyhedral and convex. Then
\begin{align}\label{eq:mixedDirichlet}
 \lambda_{k + \dim \cS (\Gamma_{\rm D})}^\Gamma \leq \lambda_k
\end{align}
holds for all $k \in \N$.
\end{thm}

\begin{proof}
Let $k \in \N$ and let $u_j$ be real-valued Dirichlet Laplacian eigenfunctions corresponding to the eigenvalues $\lambda_j$, $j = 1, \dots, k$, being pairwise orthogonal in $L^2 (\Omega)$. For $a_1, \dots, a_k, b_1, \dots, b_d \in \C$ define
\begin{align}\label{eq:PhiPsi}
 \Phi = \sum_{j = 1}^k a_j u_j \in H^2 (\Omega) \cap H_0^1 (\Omega) \quad \text{and} \quad \Psi = \sum_{j = 1}^d b_j \partial_j u_k \in H^1 (\Omega).
\end{align}
Note that by Green's identity
\begin{align*}
 \int_\Omega \nabla u_j \cdot \nabla u_l \dd x & = - \int_\Omega \Delta u_j u_l \dd x = \lambda_j \int_\Omega u_j u_l \dd x = 0, \quad j, l \in \{1, \dots, k\},~j \neq l.
\end{align*}
Note further that $- \Delta \Psi = \lambda_k \Psi$ holds in the distributional sense. With these observations and $\Phi |_{\partial \Omega} = 0$ we get
\begin{align}\label{eq:erste}
\begin{split}
 \int_\Omega |\nabla (\Phi + \Psi)|^2 \dd x &  = \sum_{j = 1}^k \int_\Omega |a_j \nabla u_j|^2 \dd x + 2 \Real \int_\Omega \nabla \Psi \cdot \overline{\nabla \Phi} \dd x + \int_\Omega |\nabla \Psi|^2 \dd x \\
 & = \sum_{j = 1}^k \lambda_j \int_\Omega |a_j u_j|^2 \dd x + 2 \lambda_k \Real \int_\Omega \Psi \overline{\Phi} \dd x + \int_\Omega |\nabla \Psi|^2 \dd x.
\end{split}
\end{align}
Moreover, for the last integral with the help of Lemma~\ref{lem:det} we obtain
\begin{align}\label{eq:istJaToll}
\begin{split}
 \int_\Omega |\nabla \Psi|^2 \dd x & = \sum_{m = 1}^d \int_\Omega \bigg| \sum_{j = 1}^d b_j \partial_{m j} u_k \bigg|^2 \dd x \\
 & = \sum_{m = 1}^d \int_\Omega \bigg( \sum_{j = 1}^d |b_j \partial_{m j} u_k |^2 + 2 \Real \bigg( \sum_{j = 1}^d \sum_{l < j} b_l \overline{b_j} (\partial_{m l} u_k) (\partial_{m j} u_k) \bigg) \bigg) \dd x \\
 & = \sum_{m = 1}^d \int_\Omega \bigg( \sum_{j = 1}^d |b_j|^2 (\partial_{j j} u_k) (\partial_{m m} u_k) \\
 & \qquad \qquad + 2 \Real \bigg( \sum_{j = 1}^d \sum_{l < j} b_l \overline{b_j} (\partial_{l j} u_k) (\partial_{m m} u_k) \bigg) \bigg) \dd x \\
 & = \sum_{m = 1}^d \int_\Omega \sum_{l, j = 1}^d b_l \overline{b_j} (\partial_{l j} u_k) (\partial_{m m} u_k) \dd x.
\end{split}
\end{align}
On the other hand, defining the $d \times d$-matrix
\begin{align*}
 B = \big(b_l \overline{b_j} \big)_{l, j = 1}^d
\end{align*}
and using integration by parts we get
\begin{align}\label{eq:super}
\begin{split}
 \lambda_k \int_\Omega |\Psi|^2 \dd x & = \lambda_k \int_\Omega \sum_{l, j = 1}^d b_l \overline{b_j} (\partial_{l} u_k) (\partial_j u_k) \dd x \\
 & = \lambda_k \int_\Omega \nabla u_k \cdot B \nabla u_k \dd x \\
 & = \int_\Omega (\Delta u_k) \diver (B \nabla u_k) \dd x \\
 & = \sum_{m = 1}^d \int_\Omega (\partial_{m m} u_k) \sum_{l, j = 1}^d b_l \overline{b_j} \partial_{l j} u_k \dd x.
\end{split}
\end{align}
Combining~\eqref{eq:istJaToll} and~\eqref{eq:super} and plugging the result into~\eqref{eq:erste} yields
\begin{align}\label{eq:dritteNeu}
\begin{split}
 \int_\Omega |\nabla (\Phi + \Psi)|^2 \dd x & = \sum_{j = 1}^k \lambda_j \int_\Omega |a_j u_j|^2 \dd x + 2 \lambda_k \Real \int_\Omega \Psi \overline \Phi \dd x + \lambda_k \int_\Omega |\Psi|^2 \dd x \\
 & \leq \lambda_k \int_\Omega |\Phi + \Psi|^2 \dd x.
\end{split}
\end{align}

In order to apply the min-max principle~\eqref{eq:minmax}, our aim is to estimate the dimension of the linear space consisting of functions of the form $\Phi + \Psi$ as in~\eqref{eq:PhiPsi} which additionally belong to $H_{0, \Gamma}^1 (\Omega)$. For this note first that 
\begin{align}\label{eq:siehMalAn}
 \dim \spann \left\{u_1, \dots, u_k, \partial_1 u_k, \dots, \partial_d u_k \right\} = k + \dim \spann \{\partial_1 u_k, \dots, \partial_d u_k\}.
\end{align}
In fact, by assumption we have $\dim \spann \{u_1, \dots, u_k\} = k$. Moreover, let
\begin{align*}
 w \in \spann \left\{u_1, \dots, u_k \right\} \cap \spann \left\{ \partial_1 u_k, \dots, \partial_d u_k \right\}.
\end{align*}
Then $w \in H_0^1 (\Omega)$ and $w = \sum_{j = 1}^d b_j \partial_j u_k$ for certain $b_1, \dots, b_d \in \C$. For a contradiction assume first that the vector $(\Real b_1, \dots, \Real b_d)^\top$ is nontrivial. Let $\Lambda$ be a face of $\partial\Omega$ such that the vector $(\Real b_1, \dots, \Real b_d)^\top$ is not tangential to $\Lambda$ and let $\tau^1, \dots, \tau^{d - 1}$ be linearly independent tangential vectors of $\Lambda$. Then the system $\{\tau^1, \dots, \tau^{d - 1}, (\Real b_1, \dots, \Real b_d)^\top\}$ is linearly independent, and due to $u_k |_\Lambda = 0$ we have
\begin{align}\label{eq:bisschenNull}
 \tau^j \cdot \nabla u_k |_\Lambda = 0, \quad j = 1, \dots, d - 1.
\end{align}
Moreover, 
\begin{align}\label{eq:nochmehrNull}
 (\Real b_1, \dots, \Real b_d)^\top \cdot \nabla u_k |_\Lambda = (\Real w) |_\Lambda = 0.
\end{align}
From~\eqref{eq:bisschenNull} and~\eqref{eq:nochmehrNull} it follows 
\begin{align*}
 \partial_\nu u_k |_\Lambda = \nu \cdot \nabla u_k |_\Lambda = 0
\end{align*}
as the constant outer unit normal $\nu$ on $\Lambda$ can be written as a linear combination of $\tau^1, \dots, \tau^{d - 1}$ and $(\Real b_1, \dots, \Real b_d)^\top$. Together with $u_k |_\Lambda = 0$, by Lemma~\ref{lem:UCP} this implies $u_k = 0$, a contradiction; thus $\Real b_1 = \dots = \Real b_d = 0$. Analogously we obtain $\Imag b_1 = \dots = \Imag b_d = 0$ and thus $w = 0$. From this we conclude~\eqref{eq:siehMalAn}. 

Let us now derive from~\eqref{eq:dritteNeu} and~\eqref{eq:siehMalAn} the assertion of the theorem. In fact, the linear space $\cS (\Gamma_{\rm D})$ is tangential to all of $\Gamma_{\rm D}$ and $u_k$ vanishes on $\Gamma_{\rm D}$. Thus 
\begin{align*}
 \sum_{j = 1}^d b_j \partial_j u_k |_{\Gamma_{\rm D}} = (b_1, \dots, b_d)^\top \cdot \nabla u_k |_{\Gamma_{\rm D}} = 0
\end{align*}
holds for all $(b_1, \dots, b_d)^\top \in \cS (\Gamma_{\rm D})$, that is, 
\begin{align}\label{eq:schonFast}
 \sum_{j = 1}^d b_j \partial_j u_k \in H_{0, \Gamma}^1 (\Omega) \quad \text{for all}~(b_1, \dots, b_d)^\top \in \cS (\Gamma_{\rm D}).
\end{align}
Next, note that $\partial_1 u_k, \dots, \partial_d u_k$ are linearly independent. For this let $b_1, \dots, b_d \in \C$ be such that 
\begin{align*}
 \sum_{j = 1}^d b_j \partial_j u_k = 0
\end{align*}
in $\Omega$ and assume for contradiction that we are off the case $b_1 = \dots = b_d = 0$. Then without loss of generality the vector $(\Real b_1, \dots, \Real b_d)^\top$ is nontrivial and the derivative of $u_k$ in the direction of this vector vanishes on all of $\Omega$. From this and $u_k |_{\partial \Omega} = 0$ it follows $u_k = 0$ on $\Omega$, a contradiction. In particular, linearly independent vectors $(b_1, \dots, b_d)^\top \in \cS (\Gamma_{\rm D})$ lead to linearly independent functions $\sum_{j = 1}^d b_j \partial_j u_k \in H_{0, \Gamma}^1 (\Omega)$, see~\eqref{eq:schonFast}. Hence
\begin{align*}
 \dim \left( \spann \{\partial_1 u_k, \dots, \partial_d u_k\} \cap H_{0, \Gamma}^1 (\Omega) \right) \geq \dim \cS (\Gamma_{\rm D}).
\end{align*}
From this and~\eqref{eq:siehMalAn} we conclude
\begin{align*}
 \dim \left( \spann \{u_1, \dots, u_k, \partial_1 u_k, \dots, \partial_d u_k \} \cap H_{0, \Gamma}^1 (\Omega) \right) \geq k + \dim \cS (\Gamma_{\rm D}).
\end{align*}
Hence~\eqref{eq:dritteNeu} together with the definition of $\Phi$ and $\Psi$ in~\eqref{eq:PhiPsi} yields
\begin{align*}
 \int_\Omega |\nabla u|^2 \dd x \leq \lambda_k \int_\Omega |u|^2 \dd x
\end{align*}
for all $u$ in a subspace of $H_{0, \Gamma}^1 (\Omega)$ of dimension $k + \dim \cS (\Gamma_{\rm D})$ or larger. This leads to the assertion of the theorem.
\end{proof}

We collect several immediate consequences of Theorem~\ref{thm:DirichletMixed}; cf.\ Figure~\ref{fig:Dirichlet}. First we consider the case of polyhedral domains and a Dirichlet boundary condition on only one face of $\Omega$. The second assertion of the following corollary makes use of Proposition~\ref{prop:monotonicity} additionally.

\begin{cor}\label{cor:DirichletMixedOneFace}
Let $\Omega \subset \dR^d$, $d \geq 2$, be a polyhedral, convex, bounded domain and let $\Gamma = \Gamma_{\rm D} \subset \Sigma$, 
where $\Sigma \subset \partial \Omega$ is either one face of $\partial \Omega$ or the union of two parallel faces. Then
\begin{align*}
 \lambda_{k + d - 1}^\Gamma  \leq \lambda_k
\end{align*}
holds for all $k \in \N$. If, in addition, $\Sigma \setminus \Gamma_{\rm D}$ has a nonempty interior then
\begin{align*}
 \lambda_{k + d - 1}^\Gamma < \lambda_k
\end{align*}
holds for all $k \in \N$.
\end{cor}

Theorem~\ref{thm:DirichletMixed} has also nontrivial implications if $\Gamma_{\rm D}$ is larger than only one face (or a pair of parallel faces). This is illustrated in the three-dimensional case in the following corollary.

\begin{cor}\label{cor:DirichletMixed3D}
Let $\Omega \subset \R^3$ be a polyhedral, convex, bounded domain. If $\Sigma_j$ is a part of $\partial \Omega$ consisting of parallel faces, $j = 1, 2$, and $\Gamma = \Gamma_{\rm D} \subset \Sigma_1 \cup \Sigma_2$ then
 \begin{align*}
 \lambda_{k + 1}^\Gamma \leq \lambda_k
\end{align*}
holds for all $k \in \N$.
\end{cor}

\begin{figure}[h]
\begin{tikzpicture}
\pgfsetlinewidth{0.8pt}
% \fill[color=lightgray] (-0.5,-1) -- (0.5,-1) -- (0.5, -0.5) -- (-0.5,-0.5) -- (-0.5,-1);
% \fill[color=lightgray] (-1,0) -- (1,0) -- (1, 0.5) -- (-1,0.5) -- (-1,0);
\fill[color=black] (-0.5,-1) -- (-1,0) -- (-1,0.5) -- (-0.5,-0.5) -- (-0.5,-1);
\draw[densely dashed] (-1,0) -- (1,0);
\draw (1,0) -- (0.5,-1) -- (-0.5, -1) -- (-1,0);
\draw (-1,0.5) -- (1,0.5) -- (0.5,-0.5) -- (-0.5, -0.5) -- (-1,0.5);
\draw (-1,0) -- (-1,0.5);
\draw (1,0) -- (1,0.5);
\draw (-0.5,-1) -- (-0.5,-0.5);
\draw (0.5,-1) -- (0.5,-0.5);
\pgfputat{\pgfxy(0.2,-0.35)}{\pgfbox[center,base]{$\Omega_1$}}
\pgfputat{\pgfxy(-1.3,-0.4)}{\pgfbox[center,base]{$\Gamma_{\rm D}$}}
\end{tikzpicture}
\qquad \qquad
\begin{tikzpicture}
\pgfsetlinewidth{0.8pt}
% \fill[color=black] (-0.5,-1) -- (0.5,-1) -- (0.5, -0.5) -- (-0.5,-0.5) -- (-0.5,-1);
% \fill[color=black] (-1,0) -- (1,0) -- (1, 0.5) -- (-1,0.5) -- (-1,0);
\fill[color=black] (-0.5,-1) -- (-1,0) -- (-1,0.5) -- (-0.5,-0.5) -- (-0.5,-1);
\fill[color=black] (0.5,-1) -- (1,0) -- (1,0.5) -- (0.5,-0.5) -- (0.5,-1);
\draw[densely dashed] (-1,0) -- (1,0);
\draw (1,0) -- (0.5,-1) -- (-0.5, -1) -- (-1,0);
\draw (-1,0.5) -- (1,0.5) -- (0.5,-0.5) -- (-0.5, -0.5) -- (-1,0.5);
\draw (-1,0) -- (-1,0.5);
\draw (1,0) -- (1,0.5);
\draw (-0.5,-1) -- (-0.5,-0.5);
\draw (0.5,-1) -- (0.5,-0.5);
\pgfputat{\pgfxy(0.2,-0.35)}{\pgfbox[center,base]{$\Omega_2$}}
\pgfputat{\pgfxy(-1.3,-0.4)}{\pgfbox[center,base]{$\Gamma_{\rm D}$}}
\end{tikzpicture}
\caption{For the example of $\Omega_1$ the inequality $\lambda_{k + 2}^\Gamma \leq \lambda_k$ holds for all $k \in \N$, see Corollary~\ref{cor:DirichletMixedOneFace}. For $\Omega_2$ one has $\lambda_{k + 1}^\Gamma \leq \lambda_k$ for all $k \in \N$, see Corollary~\ref{cor:DirichletMixed3D}.}
\label{fig:Dirichlet}
\end{figure}
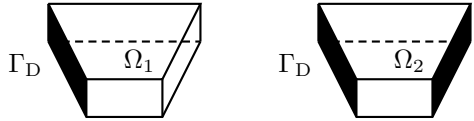

The following example demonstrates that, in general, the number $\dim \cS (\Gamma_{\rm D})$ in the eigenvalue inequality~\eqref{eq:mixedDirichlet} cannot be increased and the inequality~\eqref{eq:mixedDirichlet} is not strict. 

\begin{example}\label{ex:cube}
We consider the cube $\Omega := [0,\pi]^3 \subset \dR^3$ and suppose that $\Gamma_{\rm N} := [0,\pi]^2 \times \{0,\pi\}$, i.e., a Neumann boundary condition is imposed on two opposite faces of $\Omega$ and Dirichlet boundary conditions prevail on the rest of the boundary. In this case obviously $\dim \cS (\Gamma_{\rm D}) = 1$ and Theorem~\ref{thm:DirichletMixed} yields $\lambda_{k + 1}^\Gamma \leq \lambda_k$ for all $k \in \N$. Indeed an inequality of the form $\lambda_{k + 2}^\Gamma \leq \lambda_k$ does not hold for all $k \in \N$. In fact, the eigenvalues of the mixed problem can be calculated via separation of variables. They are given by the numbers $(n-1)^2 + m^2 + l^2$ with $n, m, l \in \dN$. On the other hand the Dirichlet Laplacian eigenvalues can be calculated analogously and have the form $n^2 + m^2 + l^2$ with $n, m, l \in \N$.  Thus
\begin{align*}
 \lambda_3^\Gamma = 5 > 3 = \lambda_1.
\end{align*}
Moreover, in this example the eigenvalue inequality $\lambda_{k + 1}^\Gamma \leq \lambda_k$ for all $k \in \N$ obtained from Theorem~\ref{thm:DirichletMixed} is not strict since we see
\begin{align*}
 \lambda_2^\Gamma = 3 = \lambda_1.
\end{align*}
\end{example}

\begin{remark}
The reasoning in the proof of Theorem~\ref{thm:DirichletMixed} can be used directly to derive the inequality
\begin{align}\label{eq:LW}
 \mu_{k + d} \leq \lambda_k \quad \text{for~all}~k \in \N
\end{align}
on any polyhedral, convex domain. Levine and Weinberger~\cite{LW86} proved this for every smooth, convex domain and extended their result to arbitrary convex domains by an approximation step. For the polyhedral case the method of the present paper is more direct.
\end{remark}

\appendix

\section{An auxiliary identity for polygonal and polyhedral domains}

The following lemma is crucial for the proof of Theorem~\ref{thm:DirichletMixed}. In the two-dimensional case it can be found in Grisvard's monograph~\cite{G85}. Below we provide a proof for the three-dimensional case and extend it afterwards to arbitrary space dimensions by reduction. We remark that the assertion of the lemma fails for general bounded, convex domains, as simple examples demonstrate.

\begin{lem}\label{lem:det}
Let $\Omega \subset \R^d$, $d \geq 2$, be a polyhedral, convex domain and let $u \in H^2 (\Omega) \cap H_0^1 (\Omega)$. Then 
\begin{align*}
 \int_\Omega (\partial_{k m} u) (\partial_{k j} u) \dd x = \int_\Omega (\partial_{m j} u) (\partial_{k k} u) \dd x
\end{align*}
holds for all $j, k, m \in \{1, \dots, d\}$.
\end{lem}

\begin{proof}
{\bf 1.} For $d = 2$ this is a special case of~\cite[Lemma~4.3.1.1--Lemma~4.3.1.3]{G85}.

{\bf 2.} Let $d = 3$. We are going to show that for $v, w \in \{\partial_1 u, \partial_2 u, \partial_3 u\}$ we have
\begin{align}\label{eq:fastSoAehnlich}
 \int_\Omega \partial_j v \partial_k w - \partial_k v \partial_j w \dd x = 0
\end{align}
for all $j, k \in \{1, 2, 3\}$. This yields the claim. We prove~\eqref{eq:fastSoAehnlich} first for $u \in C^\infty (\overline \Omega) \cap H_0^1 (\Omega)$. Note that we can assume $j \neq k$ and $v \neq w$ since otherwise the claim is satisfied trivially. Without loss of generality we assume $j = 1$ and $k = 2$. Moreover, we write $v = \partial_l u$ and $w = \partial_m u$ with $l \neq m$ and take the unique $M \in \{1, 2, 3\} \setminus \{m, l\}$. 
Note that integration by parts leads to
\begin{align*}
 \int_\Omega \partial_1 v \partial_2 w \dd x & = \int_\Omega \diver (v e_1) \partial_2 w \dd x \\
 & = \int_{\partial \Omega} \partial_2 w v e_1 \cdot \nu \dd \sigma - \int_\Omega v e_1 \cdot \nabla \partial_2 w \dd x \\
 & = \int_{\partial \Omega} \partial_2 w v \nu_1 \dd \sigma - \int_\Omega v \diver(\partial_1 w \ee_2) \dd x \\
 & = \int_{\partial \Omega} \partial_2 w v \nu_1 \dd \sigma - \bigg( \int_{\partial \Omega} v \partial_1 w e_2 \cdot \nu \dd \sigma - \int_\Omega \partial_1 w e_2 \cdot \nabla v \dd x \bigg) \\
 & = \int_{\partial \Omega} v (\partial_2 w \nu_1 - \partial_1 w \nu_2) \dd \sigma + \int_\Omega \partial_1 w \partial_2 v \dd x.
\end{align*}
Thus the assertion~\eqref{eq:fastSoAehnlich} follows if we can verify
\begin{align}\label{eq:BoundaryVanish}
 \int_{\partial \Omega} v (\partial_2 w \nu_1 - \partial_1 w \nu_2) \dd \sigma = 0.
\end{align}
In order to show this, let us denote by $\Gamma_1, \dots, \Gamma_N$ the distinct faces of $\partial \Omega$. For each $i \in \{1, \dots, N\}$ let 
\begin{align*}
 \sigma^i = \big( \sigma_1^i, \sigma_2^i, \sigma_3^i \big)^\top \quad \text{and} \quad \tau^i = \big( \tau_1^i, \tau_2^i, \tau_3^i \big)^\top
\end{align*}
be linearly independent tangential vectors of $\Gamma_i$. Since $u |_{\Gamma_i} = 0$ for each $i$, we have
\begin{align}\label{eq:tangentialId}
\begin{split}
 \sigma_l^i v |_{\Gamma_i} + \sigma_m^i w |_{\Gamma_i} + \sigma_M^i \partial_M u |_{\Gamma_i} & = 0, \\
 \tau_l^i v |_{\Gamma_i} + \tau_m^i w |_{\Gamma_i} + \tau_M^i \partial_M u |_{\Gamma_i} & = 0,
\end{split} \qquad i = 1, \dots, N.
\end{align}
Multiply the first identity in~\eqref{eq:tangentialId} by $\tau_M^i$ and the second identity by $\sigma_M^i$ and subtract the resulting equalities; this leads to
\begin{align}\label{eq:malsehen}
 \mu_1^i v |_{\Gamma_i} + \mu_2^i w |_{\Gamma_i} = 0,
\end{align}
where we have defined
\begin{align*}
 \mu_1^i = \sigma_l^i \tau_M^i - \tau_l^i \sigma_M^i \quad \text{and} \quad \mu_2^i = \sigma_m^i \tau_M^i - \tau_m^i \sigma_M^i.
\end{align*}
For each $i \in \{1, \dots, N\}$ we distinguish two cases. The first case is $\mu_2^i = 0$. If simultaneously $\mu_1^i \neq 0$ then~\eqref{eq:malsehen} yields $v |_{\Gamma_i} = 0$ and thus
\begin{align}\label{eq:zero}
 \int_{\Gamma_i} v (\partial_2 w \nu_1 - \partial_1 w \nu_2) \dd \sigma = 0.
\end{align}
If, conversely, $\mu_1^i = 0$, too, then the vectors $\tau_M^i \sigma^i$ and $\sigma_M^i \tau^i$ coincide and the linear independence of $\sigma^i$ and $\tau^i$ implies $\sigma_M^i = \tau_M^i = 0$. Using again the linear independence of $\sigma^i$ and $\tau^i$,~\eqref{eq:tangentialId} yields $v |_{\Gamma_i} = w |_{\Gamma_i} = 0$. Thus we arrive at~\eqref{eq:zero} whenever $\mu_2^i = 0$. The second case is $\mu_2^i \neq 0$. Here~\eqref{eq:malsehen} immediately yields
\begin{align}\label{eq:soSollEsSein}
 w |_{\Gamma_i} = - \frac{\mu_1^i}{\mu_2^i} v |_{\Gamma_i}.
\end{align}
Hence
\begin{align*}
\begin{split}
 \int_{\Gamma_i} v (\partial_2 w \nu_1 - \partial_1 w \nu_2) \dd \sigma & = - \frac{\mu_1^i}{\mu_2^i} \int_{\Gamma_i} v (\partial_2 v \nu_1 - \partial_1 v \nu_2) \dd \sigma \\
 & = - \frac{\mu_1^i}{2 \mu_2^i} \int_{\Gamma_i} \curl (v^2 e_3) \cdot \nu \dd \sigma \\
 & = - \frac{\mu_1^i}{2 \mu_2^i} \int_{\partial \Gamma_i} v^2 e_3 \cdot \widetilde \tau_i \dd s
\end{split}
\end{align*}
by Stokes' theorem, where $\widetilde \tau_i$ is the appropriate unit tangential vector of the piecewise linear curve $\partial \Gamma_i$ and $d s$ indicates integration along $\partial \Gamma_i$. Thus for any $i \in \{1, \dots, N\}$
we have shown
\begin{align}\label{eq:aha}
 \int_{\Gamma_i} v (\partial_2 w \nu_1 - \partial_1 w \nu_2) \dd \sigma = \eta_i \int_{\partial \Gamma_i} v^2 e_3 \cdot \widetilde \tau_i \dd s,
\end{align}
where
\begin{align*}
 \eta_i := \begin{cases} - \frac{\mu_1^i}{2 \mu_2^i}, & \text{if}~\mu_2^i \neq 0, \\ 0, & \text{if}~\mu_2^i = 0. \end{cases}
\end{align*}

We are going to conclude~\eqref{eq:fastSoAehnlich} from~\eqref{eq:aha}. Indeed, let $\Gamma_p$ and $\Gamma_q$ be any two faces of $\Omega$ which share a joint edge $\gamma = \partial \Gamma_p \cap \partial \Gamma_q$. We claim that
\begin{align}\label{eq:nochwas}
 \eta_p v |_\gamma = \eta_q v |_\gamma 
\end{align}
holds. In fact, if $\mu_2^p = 0$ then it follows as above that $v |_{\Gamma_p} = 0$ and, in particular, $v |_\gamma = 0$, which implies~\eqref{eq:nochwas} in this case. The case $\mu_2^q = 0$ is analogous. If both $\mu_2^p$ and $\mu_2^q$ are nonzero then~\eqref{eq:nochwas} is a direct consequence of~\eqref{eq:soSollEsSein}. Finally, we obtain from~\eqref{eq:aha} and~\eqref{eq:nochwas}
\begin{align*}
 \int_{\partial \Omega} v (\partial_2 w \nu_1 - \partial_1 w \nu_2) \dd \sigma & = \sum_{i = 1}^N \eta_i \int_{\partial \Gamma_i} v^2 e_3 \cdot \widetilde \tau_i \dd s \\
 & = \sum_{\gamma = \partial \Gamma_p \cap \partial \Gamma_q} \bigg( \eta_p \int_\gamma v^2 e_3 \cdot \widetilde \tau_p \dd s + \eta_q \int_\gamma v^2 e_3 \cdot \widetilde \tau_q \dd s \bigg) \\
 & = \sum_{\gamma = \partial \Gamma_p \cap \partial \Gamma_q} \eta_p \bigg( \int_\gamma v^2 e_3 \cdot \widetilde \tau_p \dd s - \int_\gamma v^2 e_3 \cdot \widetilde \tau_p \dd s \bigg) \\
 & = 0,
\end{align*}
where we have used that $\widetilde \tau_q = - \widetilde \tau_p$ holds on $\gamma$ if $\gamma = \partial \Gamma_p \cap \partial \Gamma_q$. This together with~\eqref{eq:BoundaryVanish} proves the assertion~\eqref{eq:fastSoAehnlich} for $u \in C^\infty (\overline \Omega) \cap H_0^1 (\Omega)$. For general $u \in H_0^1 (\Omega) \cap H^2 (\Omega)$ the claim follows through approximation, see, e.g.,~\cite[Corollaire~3.2]{G75}.

{\bf 3.} Let now $d > 3$. We use a dimension reduction trick from the proof of~\cite[Th\'eor\`eme~2.1]{G75}. Let $u \in H^2 (\Omega) \cap H_0^1 (\Omega)$ and let $j, k, m \in \{1, \dots, d\}$. Again we can assume that $j \neq k$ and $m \neq k$. We distinguish two cases. If $j = m$ then without loss of generality $j = m = 1$ and $k = 2$. For fixed $x_0 \in \R^{d - 2}$ we define
\begin{align*}
 \Omega_{x_0} = 
 \left\{ (x_1, x_2, x_0)^\top : (x_1, x_2, x_0)^\top \in \Omega \right\},
\end{align*}
the intersection of $\Omega$ with the plane $(0, 0, x_0)^\top + \spann \{\ee_1, \ee_2\}$. Then for almost all $x_0 \in \R^{d - 2}$ we have
\begin{align*}
 u |_{\Omega_{x_0}} \in H^2 (\Omega_{x_0}) \cap H_0^1 (\Omega_{x_0})
\end{align*}
and it follows from the result of the lemma for $d = 2$
\begin{align*}
 \int_{\Omega_{x_0}} (\partial_{1 2} u) (\partial_{1 2} u) \dd x = \int_{\Omega_{x_0}} (\partial_{1 1} u) (\partial_{2 2} u) \dd x
\end{align*}
for almost all $x_0 \in \R^{d - 2}$. Now integration over $x_0$ yields the claim. In the second case $j \neq m$ we apply the same procedure, intersecting $\Omega$ with shifts of $\spann \{\ee_j, \ee_k, \ee_m\}$ and using the result of the lemma for $d = 3$. This completes the proof.
\end{proof}

\section*{Acknowledgements}
VL gratefully acknowledges financial support by the grant No.\ 14-06818S of the Czech Science Foundation (GA\v{C}R). Moreover, he is grateful to Hamburg University of Technology for the hospitality during a visit in July 2016, when a part of this paper was written.

% *******************************************************************

\end{document}